\documentclass[12pt]{article}
\usepackage[utf8]{inputenc}
\usepackage[english]{babel}
\usepackage{amsmath}
\usepackage{amsthm}
\usepackage{amsfonts}
\usepackage{amssymb}
\usepackage{mathrsfs}
\usepackage[margin=20mm]{geometry}
\usepackage{enumerate}
\usepackage{color}
\usepackage{hyperref}

\newcommand{\class}{\mathcal{C}}
\newcommand{\sphere}{\mathbb{S}}
\DeclareMathOperator{\vol}{vol}
\newcommand{\R}{\mathbb{R}}
\newcommand{\supp}{\mathop{\mathrm{supp}}}
\newtheorem{teo}{Theorem}[section]
\newtheorem{lem}[teo]{Lemma}
\newtheorem{rmk}[teo]{Remark}

\newtheorem{cor}[teo]{Corollary}
\newtheorem{prop}[teo]{Proposition}

\title{Some remarks on Petty projection of log-concave functions}
\date{\today}

\author{Leticia Alves da Silva\thanks{Partially supported by FAPERJ, project reference 236508 Brazil.}, 
Bernardo Gonz\'alez Merino%
\thanks{Partially supported by 
 MICINN project PGC2018-094215-B-I00 Spain.
},
Rafael Villa%
\footnotemark[2]
}

\AtEndDocument{\bigskip{\footnotesize%
  \textsc{IFMG Campus Bambu\'i, Instituto Federal de Minas Gerais, - Faz, Varginha -, Rodovia Bambu\'i/Medeiros - km 05. CP 05 - Bambu\'i - MG - CEP:38900-000, Brazil} \par
  \textit{E-mail address} L.~Alves da Silva: \texttt{leticia.alves@ifmg.edu.br} \par
  \addvspace{\medskipamount}
  \textsc{Universidad de Murcia, Facultad de Inform\'atica, Departamento de Ingenier\'ia y Tecnolog\'ia de Computadores, \'Area de Matem\'atica Aplicada, 30100-Murcia, Spain} \par
  \textit{E-mail address} B.~Gonz\'alez Merino: \texttt{bgmerino@um.es} \par
  \textsc{Universidad de Sevilla, Departamento de An\'alisis Matem\'atico, Apartado de Correos 1160, Sevilla, 41080, Spain} \par
  \textit{E-mail address} R.~Villa: \texttt{villa@us.es}
}}

\begin{document}

\maketitle
	


\abstract{In this note we study the Petty projection of a log-concave function, which has been recently introduced in \cite{fang2018lyz}. The aim
of this note is to report a mistake in Theorem 5.2 of \cite{fang2018lyz} and to give correct new inequalities involving this new notion.
}
	
\section{Introduction}


Let $K\subset\mathbb R^n$ be a \emph{convex body}, i.e., a convex and compact set with non-empty interior, whose
boundary is denoted by $\partial K$.
Moreover, let $\mathcal K^n$ be the set of all convex bodies in $\mathbb R^n$. For these and most of the forthcoming definitions and ideas on Convex Geometry, we recommend the books \cite{schneider2014convex} and \cite{artstein2015asymptotic}.

If the origin is an interior point of $K$, the polar body $K^\circ$ of $K$ is
$$
K^\circ=
\{x\in\R^n: \langle x,y\rangle\le1\text{ for any }y\in K\},
$$
which is also a convex body with the origin in its interior.

A convex body $K\in\mathcal K^n$ is uniquely defined by its \emph{support function}, defined by
\begin{equation}
\label{support}
h_K(x)=\sup\{\langle x,y\rangle:y\in K\}.
\end{equation}

For $C\subset\mathbb R^n$ of affine dimension $k\in\{1,\dots,n\}$, we denote by $\mathrm{vol}_k(C)$
its volume measured inside the affine hull of $C$, and moreover we also write $\mathrm{vol}(C)=\mathrm{vol}_n(C)$.

The mixed volume of two convex bodies $K$ ($n-1$ times) and $L$ can be defined by
$$
V_1(K,L)=
\frac1n
\lim_{\varepsilon\to0^+}\frac{\vol(K+\varepsilon L)-\vol(K)}{\varepsilon}.
$$

There is a unique finite measure $S(K,\cdot)$ on the unit euclidean sphere $\sphere^{n-1}$, called the \emph{surface area} of $K$,  so that
$$
V_1(K,L)=\frac1n\int_{\sphere^{n-1}}h_L(v)\,dS(K,v),
$$
(cf.~\cite{lutwak1993the}). When $K$ has a $\class^2$ boundary $\partial K$ with positive curvature, the density of $S(K,\cdot)$ with respect to the Lebesgue measure on $\sphere^{n-1}$ is the reciprocal of the Gauss curvature of $\partial K$.

 For any $x\in\R^n$ we let $x^\bot$ be
the $(n-1)$-subspace orthogonal to $x$, and let $P_{x^\bot}C$ be the orthogonal projection of $C$ onto $x^\bot$.
Then the \emph{projection body} $\Pi K$ of $K\in\mathcal K^n$ is the centrally symmetric convex body
given by its support function
$$
h_{\Pi K}(u)=\vol_{n-1}(P_{u^\bot}K),
$$
for every $u\in\sphere^{n-1}$. Using standard properties of the mixed volume $V_1(\cdot,\cdot)$ (cf.~\cite{schneider2014convex}) it is easy to see that
\begin{equation}\label{eq:supportProj}
h_{\Pi K}(u)=\frac12\int_{\mathbb S^{n-1}}|\langle u,v\rangle|dS(K,v).
\end{equation}
In fact,
$$
\int_{\mathbb S^{n-1}}|\langle u,v\rangle|dS(K,v)=
\int_{\sphere^{n-1}}h_{L_u}(v)\,dS(K,v),
$$
where $L_u=[-u,u]=\{tu:t\in[-1,1]\}$. Using Fubini's formula,
$$
{\vol(K+\varepsilon L_u)-\vol(K)}=2\varepsilon\int_{P_{u^\bot}K}dx'=2\varepsilon\vol_{n-1}({P_{u^\bot}K}).
$$
Then
$$
n V_{1}(K,L_u)
=\lim_{\varepsilon\to 0^+}\frac{\vol(K+\varepsilon L_u)-\vol(K)}{\varepsilon}=
2\vol_{n-1}({P_{u^\bot}K})
=2h_{\Pi K}(u).
$$

Finally, the \emph{polar projection body} $\Pi^\circ K$ is the polar body of $\Pi K$.



A function $f:\R^n\rightarrow[0,\infty)$ is \emph{log-concave} if $\log f$ is concave, i.e., if
\[
f((1-\lambda)x+\lambda y)\geq f(x)^{1-\lambda}f(y)^\lambda,
\]
for every $x,y\in\mathbb R^n$, $\lambda\in(0,1)$. Then $f=e^{-\varphi}$ for a convex function $\varphi:\R^n\rightarrow[-\infty,\infty)$.
Moreover, let $\mathcal F(\mathbb R^n)=\{f\text{ log-concave with }f\in L^1(\mathbb R^n)\}$.

Two typical embeddings of all convex bodies onto the set $\mathcal F(\mathbb R^n)$ are given by
the mappings that identify $K$ either with the \emph{characteristic function} $\chi_K(x)=e^{-I^\infty_K(x)}$
or the \emph{exponential gauge} $e^{-\Vert x\Vert_K}$ of $K$, where
\[
I^{\infty}_{K}(x) = \left\{
\begin{array}{lr}
0 & \text{if } \quad x \in K\\
\infty & \text{otherwise.}
\end{array} \right.
\quad\text{and}\quad
\Vert x\Vert_K=\inf\{t>0:x\in tK\}.
\]


Considering the definition of $h_K$ given by (\ref{support}), we can write
\begin{equation}\label{eq:supportfuction}
h_K(x)=\sup \{\langle x,y\rangle-I_K^\infty(y):y\in\R^n\}=(I_K^\infty)^*(x),
\end{equation}
where, for a convex function $\varphi$,
$$
\varphi^*(x)=\sup \{\langle x,y\rangle-\varphi(y):y\in\R^n\}
$$
is the so called \emph{Legendre transform} of $\varphi$ (cf.~\cite{klartag2005geometry}). As
$$
\chi_K=e^{-I_K^\infty}\text{ and }h_K=(I_K^\infty)^*,
$$
it is natural to define the \emph{support function} of a log-concave function $f=e^{-\varphi}$ as
$$
h_f=\varphi^*,
$$
(cf. \cite{rotem2012mean}).
Note  that $h_K^*=I^\infty_K$ for every $K\in\mathcal K^n$ (cf.~\cite{klartag2005geometry}).


In order to define the polar function of a function $f=e^{-\varphi}$ as a log-concave function, it is natural to search for a transformation $T$ between convex functions so that $f^\circ=e^{-T\varphi}$. Since $(K^\circ)^\circ=K$ and if $K_1\subset K_2$ then $K_1^\circ\supset K_2^\circ$, we have to ask for $T$ to verify $T^2$ to be the identity, and if $\varphi_1\le\varphi_2$, then $T\varphi_1\ge T\varphi_2$. From \cite{artstein2009The}, these properties characterize the Legendre transform, so $T\varphi=\varphi^*$.

As a consequence, for any log-concave $f:\R^n\rightarrow[0,+\infty)$ with $f=e^{-\varphi}$, its \emph{polar function} $f^\circ$  is defined by $f^\circ=e^{-\varphi^*}$ (cf.~\cite{artstein2009The}).
With this definition, if $f\in\mathcal (\mathbb R^n)$ with $0\in\mathrm{int}(\mathrm{supp}f)$, and $f^\circ(x_M)=\|f\|_\infty$ for some $x_M\in\R^n$, then $f^\circ\in\mathcal F(\mathbb R^n)$ too
(see Theorem \ref{thm:IntegrPolar} below).
Note that $f^\circ=e^{-h_f}$.

To define the analogue definition of $\Pi f$ for a log-concave $f$, firstly defined in \cite{fang2018lyz}, we take into account the equality for a convex body $K$
$$
\int_{\sphere^{n-1}}|\langle u,v\rangle|\,dS(K,v)
=
\int_{\R^n}|\langle\nabla\chi_K(x),u\rangle|\,dx,
$$
for $u\in\sphere^{n-1}$ (see \cite{zhang1999affine}, or Proposition \ref{prop:identities} \ref{prop:proj}) ). We may now generalize (\ref{eq:supportProj}) to define  the \emph{Petty projection function} $\Pi f$ of $f$ given its support function
\[
h_{\Pi f}(u)=\frac12\int_{\R^n}|\langle\nabla f(x),u\rangle|\,dx,
\]
(see  \cite{fang2018lyz}). Note that, by the chain rule, if $f=e^{-\varphi}$, then $\nabla f=-f\nabla\varphi$, and the previous definition admits the form
$$
h_{\Pi f}(u)=\frac12\int_{\supp f}|\langle\nabla \varphi(x),u\rangle|f(x)\,dx.
$$
In particular, for any $f\in\mathcal F(\mathbb R^n)$, the \emph{polar projection function}
is given by $\Pi^\circ f=(\Pi f)^\circ$.

\section{Properties and main result}

The main result here serves as a correction to \cite[Thm.~5.2]{fang2018lyz} and introduces a lower bound for the integral of $\Pi^\circ f$.
Let us denote by $B^n_2$ the \emph{n-dimensional Euclidean unit ball}, $\mathbb S^{n-1}$ its boundary, and let $\omega_n=\mathrm{vol}(B^n_2)$ be its volume. Moreover, let $|x|=\sqrt{x_1^2+\cdots+x_n^n}$ be the \emph{Euclidean norm} for every $x=(x_1,\dots,x_n)\in\mathbb R^n$.

\begin{teo}\label{thm:polar}
Let $f\in\mathcal F(\mathbb R^n)$. Then
\[
\left(\int_{\mathbb R^n}|\nabla f(x)|dx\right)^n\int_{\mathbb R^n}\Pi^\circ f(z)dz\geq \omega_n n!\left(\frac{n\omega_n}{\omega_{n-1}}\right)^n.
\]
Moreover, equality holds if there exists $g:[0,\infty)\rightarrow[0,\infty)$, $g\in\mathcal F(\mathbb R^1)$,
such that $f(x)=g(|x|)$ for every $x\in\mathbb R^n$.
\end{teo}

In the next proposition we collect some useful computations needed in this paper, refereed to characteristic functions and exponential gauges of convex bodies.

\begin{prop}\label{prop:identities}
Let $K\in\mathcal K^n$. Then we have that
\begin{enumerate}[i)]
\item $h_{\chi_K}=h_K$.
\item \label{prop:hOfExp} $h_{e^{-\Vert\cdot\Vert_K}}=I^\infty_{K^\circ}$.
\item \label{prop:proj} $h_{\Pi \chi_K}=h_{\Pi K}$.
\item \label{prop:projcharact}$\Pi\chi_K=e^{-I^\infty_{\Pi K}}=\chi_{\Pi K}$.
\item \label{prop:polar_proj} $\Pi^\circ\chi_K=e^{-h_{\Pi K}}$.
\item \label{eq:suppprojexp}$h_{\Pi e^{-\Vert\cdot\Vert_K}}=(n-1)!h_{\Pi K}$.
\item \label{prop:vol} $\int_{\mathbb R^n}e^{-\Vert x\Vert_K}dx=n!\mathrm{vol}(K)=n!\int_{\mathbb R^n}\chi_K(x)dx$.
\item \label{eq:intprojpolarexp}$\int_{\mathbb R^n}\Pi^\circ (e^{-\Vert \cdot\Vert_K})(x)dx=n!\Gamma(n)^{-n}\mathrm{vol}(\Pi^\circ K)=\Gamma(n)^{-n}\int_{\mathbb R^n}\Pi^\circ (\chi_K)(x)dx$.
\end{enumerate}
\end{prop}
Let us observe that Theorem 5.2 in \cite{fang2018lyz} is not correct. Indeed, using Proposition \ref{prop:identities},
one can verify that if $f=\chi_{tB^n_2}$, for some $t>0$, then the Theorem 5.2 in \cite{fang2018lyz} becomes
$(1-\log t)^{n-1}\leq C(n)$, for some constant $C(n)>0$ only depending on the dimension $n$. Later on, we will discuss how to
correct those bounds for the integral of the Petty projection function.

One can also bound from above the term $\int|\nabla f|$ by means of an \emph{entropic function} under some extra assumptions;
indeed, if $a\chi_{B^n_2}\leq f$ for some $f\in\mathcal F(\mathbb R^n)$ and $a>0$, then
\begin{equation}\label{eq:entropic}
\int_{\mathbb R^n}|\nabla f(x)|dx\leq n\int_{\mathbb R^n}f(y)dy+\int_{\mathbb R^n}f(z)\log\frac{f(z)}{a\Vert f\Vert_\infty}dz,
\end{equation}
with equality if $f=a\chi_{B^n_2}$ (cf.~\cite{alonso2017john}, see also \cite{bobkov2012reverse}).

The quantity $\mathrm{vol}(K)^{n-1}\mathrm{vol}(\Pi^\circ K)$ is an affine invariant,
its maximum value is provided by \emph{Petty's projection inequality} \cite{petty1971isoperimetric}, with equality if and only if $K$ is an ellipsoid,
and its minimum value is given by \emph{Zhang's inequality} \cite{zhang1991restricted}, with equality if and only if $K$ is a simplex:
\begin{equation}\label{eq:Petty and Zhang}
\frac{{2n\choose n}}{n^n}\leq\mathrm{vol}(K)^{n-1}\mathrm{vol}(\Pi^\circ K)\leq\left(\frac{\omega_n}{\omega_{n-1}}\right)^n.
\end{equation}
For any $f\in\mathcal F(\mathbb R^n)$, $f=e^{-\varphi}$, let $\Pi_bf$ be the \emph{Petty projection body} of $f$, which is the convex body
whose support function is given by
\[
h_{\Pi_bf}(y)=\int_{supp\,f}|\langle \nabla\varphi(x),y\rangle|f(x)dx,\quad\text{i.e.}\quad h_{\Pi_bf}=2h_{\Pi f}.
\]
In order to avoid future confusion, here we have changed the original name also
given by Fang and Zhou \cite{fang2018lyz} (they used the name $\Pi f$, and we insert the subindex $b$ to stress that it is a body).
Its polar $\Pi^\circ_bf$ was firstly introduced in \cite{alonso2017john}, and here once more we change the old naming $\Pi^\circ f$ by $\Pi^\circ_bf$,
and it is the unit ball of the norm given by
\[
\Vert x\Vert_{\Pi_b^\circ f}=\int_{\mathbb R^n}|\langle\nabla f(x),y\rangle|dx.
\]
Since $f=e^{-\varphi}$, due to $\nabla f(x)=-e^{-\varphi(x)}\nabla\varphi(x)$,
we have that $\Pi^\circ_bf=(\Pi_bf)^\circ$, as one may expect.
The right-hand side of \eqref{eq:Petty and Zhang} was extended to functional settings
by Zhang \cite{zhang1999affine} and it is also known as the affine Sobolev inequality, whereas the left-hand side of \eqref{eq:Petty and Zhang}
was recently extended to log-concave functions in \cite{alonso2018zhang},

\begin{equation}\label{eq:functional petty zhang}
\frac{2^{-n}}{n!}\Vert f\Vert_1^{-n-1}\int_{\mathbb R^n}\int_{\mathbb R^n}\min\{f(x),f(y)\}dxdy\leq\mathrm{vol}(\Pi^\circ_bf)\leq\left(\frac{\omega_n}{2\omega_{n-1}}\right)^n\Vert f\Vert_{\frac{n}{n-1}}^{-n}.
\end{equation}

Moreover, equality holds on the right-hand side if and only if $\frac{f}{\Vert f\Vert_\infty}=\chi_{AB^n_2}$, for any regular $A\in\mathbb R^{n\times n}$,
and on the left hand side if and only if $\frac{f}{\Vert f\Vert_\infty}=e^{-\Vert \cdot\Vert_S}$, for any simplex $S\in\mathcal K^n$, with $0 \in S$.
One can immediately verify that

\begin{equation}\label{eq:polar_b}
\Pi^\circ f(y)=e^{-h_{\Pi f}(y)}=e^{-\frac12h_{\Pi_bf}(y)}=e^{-\frac12\Vert y\Vert_{\Pi^\circ_bf}},
\end{equation}

and thus, \eqref{eq:functional petty zhang} can be used to give optimal bounds of the integral of $\Pi^\circ f$ for any $f\in\mathcal F(\mathbb R^n)$.
\begin{prop}\label{prop:Integral_and_volume}
Let $f\in\mathcal F(\mathbb R^n)$. Then
\[
\int_{\mathbb R^n}\Pi^\circ f(x)dx=2^nn!\mathrm{vol}(\Pi^\circ_bf).
\]
\end{prop}
After writing this note, Fang and Zhou have told us in personal communication that they also noticed their mistake; however, it
seems that they have amended it replacing it by the right-hand side of \eqref{eq:functional petty zhang} and using Proposition
\ref{prop:Integral_and_volume}.

\section{Proofs}

We start this section by proving Proposition \ref{prop:identities}.

\begin{proof}[Proof of Proposition \ref{prop:identities}]
\begin{enumerate}[i)]
\item See \eqref{eq:supportfuction}.

\item It is a direct consequence of $\Vert\cdot\Vert_K^*=I^\infty_{K^\circ}$.

\item
Here we use a similar argument to the one exhibited in \cite[\S4]{zhang1999affine}.
Let us denote by $d(x,A)$ the Euclidean distance from a point $x \in \mathbb{R}^n$ to a set $A \subset \mathbb{R}^n$.
Let $\varepsilon > 0$ and define
\begin{equation}\nonumber 
f_{\varepsilon}(x) = \left\{
\begin{array}{lr}
0 & \text{if }  d(x,K) \geq \varepsilon\\
1 - \frac{d(x, K)}{\varepsilon} & \text{if }  d(x,K) < \varepsilon
\end{array} \right.,
\end{equation}
If $d(x, K)>0$ for some $x\in\mathbb R^n$, then there exists a unique $x' \in \partial K$ such that $d(x, K) = |x - x'|$. Let $\nu(x')=\frac{x-x'}{|x - x'|}$ be the outer normal of $K$ at $x'$ and let $D_{\varepsilon} = \{x \in \mathbb{R}^n: 0< d(x, K) < \varepsilon\}$. Then
\begin{equation}\nonumber 
\nabla f_{\varepsilon}(x) = \left\{
\begin{array}{lr}
-\varepsilon^{-1}\nu(x') & \text{if } x \in D_{\varepsilon}\\
0 & \text{otherwise}\\
\end{array} \right.,
\end{equation}
from which
\begin{equation*}
\frac{1}{2}\int_{\mathbb{R}^n}|\langle \nabla f_{\varepsilon}(x), y \rangle|dx = \frac{1}{2}\int_{D_{\varepsilon}}|\langle \varepsilon^{-1}\nu(x'), y\rangle|dx
 = \frac{\varepsilon^{-1}}{2}\int_{D_{\varepsilon}}|\langle \nu(x'), y\rangle|dx.
\end{equation*}

When $\varepsilon \rightarrow 0$, we have that
\begin{equation*}
	\frac{\varepsilon^{-1}}{2}\int_{D_{\varepsilon}}|\langle \nu(x'), y\rangle|dx \rightarrow \frac{1}{2}\int_{\partial K}|\langle \nu (x'), y \rangle|d\sigma(\partial K,x'),
\end{equation*}
where $d\sigma(\partial K,\cdot)$ is the surface area element of $\partial K$. Since $\lim_{\varepsilon\rightarrow 0}f_{\varepsilon}=\chi_K$
and by (\ref{eq:supportProj}) we can conclude that
\begin{align*}
	h_{\Pi \chi_K}(y) & = \lim_{\varepsilon \rightarrow 0}\frac{1}{2}\int_{\mathbb{R}^n}|\langle \nabla f_{\varepsilon}(x), y \rangle|dx \\
	& = \frac{1}{2}\int_{\partial K}|\langle \nu (x'), y \rangle|d\sigma(\partial K,x') \\
	& = \frac{1}{2}\int_{\mathbb{S}^{n-1}}|\langle u, y \rangle|dS(K, u) \\
	& = h_{\Pi K}(y).
\end{align*}

\item Since by definition $\Pi\chi_K(x)=e^{-h_{\Pi\chi_K}^*(x)}$, by \ref{prop:proj}) we can conclude that
\[
\Pi\chi_K(x)=e^{-h^*_{\Pi K}(x)}=e^{-I^\infty_{\Pi K}(x)}=\chi_{\Pi K}(x).
\]

\item Using \ref{prop:proj}) we have that $\Pi^{\circ}\chi_K(x) = e^{-h_{\Pi\chi_K}(x)} = e^{-h_{\Pi K}(x)}$, as desired.

\item See \cite[Prop.~5.1]{fang2018lyz}.

\item By definition, $\int_{\mathbb R^n}\chi_K(x)dx=\mathrm{vol}(K)$. Second,
\[
\begin{split}
\int_{\mathbb R^n}e^{-\Vert x\Vert_K}dx & = \int_0^1\mathrm{vol}(\{x\in\mathbb R^n:e^{-\Vert x\Vert_K}\geq t\})dt\\
& =\int_0^1\mathrm{vol}(\{x\in\mathbb R^n:\Vert x\Vert_K\leq -\log t\})dt\\
& =\mathrm{vol}(K)\int_0^1\left(\log \frac{1}{t}\right)^ndt\\
& =\mathrm{vol}(K)\int_0^\infty s^ne^{-s}ds\\
&=\Gamma(n+1)\mathrm{vol}(K).
\end{split}
\]

\item On the one hand, using \ref{prop:polar_proj}) and \ref{prop:vol}) we immediately get that
\[
\int_{\mathbb R^n}\Pi^\circ\chi_K(x)dx=\int_{\mathbb R^n}e^{-h_{\Pi K}(x)}dx=\int_{\mathbb R^n}e^{-\Vert x\Vert_{\Pi^\circ K}(x)}dx
=n!\mathrm{vol}(\Pi^\circ K).
\]
On the other hand, using \ref{eq:suppprojexp}) and \ref{prop:vol}) and the 1-homogeneity of the support function we can conclude that
\[
\begin{split}
\int_{\mathbb R^n}\Pi^\circ(e^{-\Vert\cdot\Vert_K})(x)dx & = \int_{\mathbb R^n}e^{-h_{\Pi e^{-\Vert\cdot\Vert_K}}(x)}dx\\
& =\int_{\mathbb R^n}e^{-\Gamma(n)h_{\Pi K}(x)}dx\\
& =\Gamma(n)^{-n}\int_{\mathbb R^n}e^{-h_{\Pi K}(x)}dx\\
&=\Gamma(n)^{-n}n!\mathrm{vol}(\Pi^\circ K).
\end{split}
\]
\end{enumerate}
\end{proof}

We now prove Theorem \ref{thm:polar}. The main ingredients of it are the integration by polar
coordinates and the \emph{Jensen inequality} \cite{artstein2015asymptotic}, which states that if $(X, \Sigma, \mu)$ is a probability space, then for any convex function $\varphi:\mathbb R\rightarrow\mathbb R$ and any $\mu$-integrable function $f:X \rightarrow\mathbb R$, we have that
\[
\varphi\left(\int_{X}f(x)d\mu(x)\right)\leq\int_{X}\varphi\circ f(x)d\mu(x),
\]
and moreover, equality holds if and only if either $\varphi$ is affine or $f$ is independent of $x$.
One can compare the proof below to the one in \cite[Thm.~5.2]{fang2018lyz}, where we have detected mistakes in (5.17) at the change of variables and
at the application of Jensen inequality.

\begin{proof}[Proof of Theorem \ref{thm:polar}]
Let $f=e^{-\varphi}$. Since $\nabla f(x)=-f(x)\nabla\varphi(x)$ and using polar coordinates, we can write
\[
\begin{split}
\int_{\mathbb R^n}\Pi^\circ f(z)dz & =\int_{\mathbb R^n}e^{-\frac12\int_{\mathbb R^n}|\langle \nabla f(x),z\rangle|dx}dz \\
&=n\omega_n\int_0^\infty\int_{\mathbb S^{n-1}}e^{-\frac{r}{2}\int_{\mathbb R^n}|\langle\nabla f(x),u\rangle|dx}r^{n-1}d\mu(u)dr,
\end{split}
\]
where $\mu$ is the uniform probability measure in $\mathbb S^{n-1}$.
Since $e^x$ is convex, Jensen inequality implies that
\[
\exp\left(\int_{\mathbb S^{n-1}}-\frac r2 \int_{\mathbb R^n}|\langle\nabla f(x),u\rangle|dxd\mu(u)\right)
\leq
\int_{\mathbb S^{n-1}}\exp\left(-\frac{r}{2}\int_{\mathbb R^n}|\langle\nabla f(x),u \rangle|dx\right)d\mu(u).
\]
Using Fubini and using the fact that
\[
\int_{\mathbb S^{n-1}}|\langle v,u\rangle|d\mu(u)=\frac{2\omega_{n-1}}{n\omega_n}|v|,
\]
for any $v\in\mathbb R^n$, then
\[
\begin{split}
\int_{\mathbb R^n}\Pi^\circ f(z)dz & \geq n\omega_n\int_0^\infty e^{-\frac r2 \int_{\mathbb R^n}\int_{\mathbb S^{n-1}}|\langle\nabla f(x),u\rangle|d\mu(u)dx}r^{n-1}dr \\
&= n\omega_n\int_0^\infty e^{-\frac r2 \int_{\mathbb R^n}\frac{2\omega_{n-1}}{n\omega_n}|\nabla f(x)|dx}r^{n-1}dr\\
&=n\omega_n\int_0^\infty e^{-\frac{\omega_{n-1}}{n\omega_n}\Vert\nabla f\Vert_1 r}r^{n-1}dr.
\end{split}
\]
Letting $t=\frac{\omega_{n-1}}{n\omega_n}\Vert\nabla f\Vert_1r=ar$, then $dt=adr$ and
\[
\int_0^\infty e^{-ar}r^{n-1}dr=\int_0^\infty e^{-t}t^{n-1}a^{1-n}a^{-1}dt=a^{-n}\Gamma(n).
\]
We can thus conclude that
\[
\begin{split}
\int_{\mathbb R^n}\Pi^\circ f(z)dz & \geq n!\omega_n\left(\frac{\omega_{n-1}}{n\omega_n}\Vert\nabla f\Vert_1\right)^{-n}\\
& =n!\omega_n\left(\frac{\omega_{n-1}}{n\omega_n}\right)^{-n}\left(\int_{\mathbb R^n}|\nabla f(x)|dx\right)^{-n}.
\end{split}
\]

In the equality case, there must be equality in the inequality above. Hence, by Jensen's equality case,
we must have that $\int_{\mathbb R^n}|\langle\nabla f(x),u\rangle|dx$ is independent of $u\in\mathbb S^{n-1}$. In particular,
if $f(x)=g(|x|)$ for some $g:[0,\infty)\rightarrow[0,\infty)$ log-concave and every $x\in\mathbb R^n$, as desired.
\end{proof}

A geometrical consequence of Theorem \ref{thm:polar} is the following result (cf.~\cite{zhang1999affine}), which relates the
\emph{surface area measure} $S(K)$ of a $K\in\mathcal K^n$ with the volume $\mathrm{vol}(\Pi^\circ K)$, and can be also
obtained by H\"older inequality in \eqref{eq:Petty and Zhang}.
\begin{cor}
Let $K\in\mathcal K^n$. Then
\[
S(K)^n\mathrm{vol}(\Pi^\circ K)\geq\omega_n\left(\frac{n\omega_n}{\omega_{n-1}}\right)^n.
\]
Moreover, equality holds if $K=B^n_2$.
\end{cor}

\begin{proof}
Let us particularize Theorem \ref{thm:polar} taking $f(x)=e^{-\Vert x\Vert_K}$.
If we denote by $d\sigma(\partial K,\cdot)$ the surface area element of $K$, then
\begin{align*}
\int_{\mathbb{R}^n}|\nabla f(x)|dx & = \int_{\mathbb{R}^n}e^{-||x||_K}|\nabla||x||_K|dx \\
 & = \int_{0}^{\infty}\int_{t\partial K}e^{-||x||_K} |\nabla||x||_K|\cdot|\nabla||x||_K|^{-1}d\sigma(t\partial K, x)dt \\
& = \int_{0}^{\infty}\int_{t\partial K}e^{-||x||_K}d\sigma(t\partial K, x)dt \\
& = \int_{0}^{\infty}\int_{\partial K}e^{-t} t^{n-1}d\sigma(\partial K, y)dt \\
& =  \int_{0}^{\infty}e^{-t}t^{n-1}dt\int_{\partial K}d\sigma(\partial K, y) \\
& = \int_{0}^{\infty}e^{-t}t^{n-1}dt S(K)\\
& = \Gamma(n)S(K).
\end{align*}
This, together with \ref{eq:intprojpolarexp}) in Proposition \ref{prop:identities}, imply that
\[
\begin{split}
\frac{n!}{(n-1)!^n}\mathrm{vol}(\Pi^\circ K) & = \int_{\mathbb R^n}\Pi^\circ(e^{-\Vert\cdot\Vert_K})(x)dx\\
& \geq n!\omega_n\left(\frac{\omega_{n-1}}{n\omega_n}\right)^{-n}\left(\int_{\mathbb R^n}|\nabla(e^{-\Vert \cdot\Vert_K})(x)|dx\right)^{-n}\\
& =n!\omega_n\left(\frac{\omega_{n-1}}{n\omega_n}\right)^{-n}(n-1)!^{-n}S(K)^{-n},
\end{split}
\]
as desired.

Equality holds if $e^{-\Vert x\Vert_K}$ is independent of $x\in\mathbb S^{n-1}$, for instance, if $K=B^n_2$.
\end{proof}

Now we show Proposition \ref{prop:Integral_and_volume}.

\begin{proof}[Proof of Proposition \ref{prop:Integral_and_volume}]
As a consequence of (\ref{eq:polar_b}) and \ref{prop:vol}) in Proposition \ref{prop:identities}, we obtain that
\[
\int_{\mathbb R^n}\Pi^\circ f(x)dx =\int_{\mathbb R^n}e^{-h_{\Pi f}(x)}dx
 =\int_{\mathbb R^n}e^{-\frac12\Vert x\Vert_{\Pi^\circ_bf}}dx
= 2^n\int_{\mathbb R^n}e^{-\Vert x\Vert_{\Pi^\circ_bf}}dx
=2^nn!\mathrm{vol}(\Pi^\circ_bf).
\]
\end{proof}

\section{Integrability of log-concave functions}

In this section we characterize the integrability of log-concave functions in terms of the value of $f$ over all possible rays.
Other characterizations of the integrability of log-concave functions were given in \cite{cordero2015moment}.
Before stating the next result, we would like to remember that for any log-concave function $f:\mathbb R^n\rightarrow[0,\infty)$,
the function
\[
g(x)=\left\{\begin{array}{cc}f(x) & \text{if }x\in\mathrm{int}(\mathrm{supp }f),\\
\limsup_{y\rightarrow x,y\in\mathrm{int}(\mathrm{supp }f)}f(y) & \text{if }x\in\partial\,\mathrm{supp }f,\\
0 & \text{otherwise} \end{array}\right.
\]
is log-concave, continuous on its support $\supp g=\supp f=\overline{\{x\in\mathbb R^n:f(x)>0\}}$, and has $\int g=\int f$ (see \cite[Lem.~2.1]{colesanti2006functional}). Thus, we can always replace $f$ by $g$ and hence we can always extend continuously $f$ to its support.

\begin{lem}\label{lem:integrable}
Let $f:\mathbb R^n\rightarrow [0,\infty)$ be log-concave with $f(x_M)=\|f\|_\infty$ for some $x_M\in\mathbb R^n$. Then the following are equivalent:
\begin{enumerate}
\item[(1)] $f$ is integrable. 
\item[(2)] Either $\int f=0$ or there exists no ray $R=x_0+\mathbb R_+u$, $x_0,u\in\mathbb R^n$, $u\neq 0$, for which $f|_R=c>0$.
\end{enumerate}
\end{lem}

\begin{proof}
We first prove (1) implies (2). Let us suppose that $\int f>0$. Hence, the function $f$ is non-zero in an open ball $B(x_0,r)$,
for some $x_0\in\mathbb R^n$ and $r>0$.
By continuity of $f$ in $B(x_0,r)$, let us suppose that $f(x)\geq \alpha$, for every $x\in B(x_0,r)$ and for some $\alpha>0$.
Moreover, for the sake of contradiction, let us suppose that there
exists a ray $R=y_0+\mathbb R_+u$, $y_0,u\in\mathbb R^n$, $u\neq 0$, such that $f|_R=c>0$. Let $z_0\in U=B(x_0,r)\cap(x_0+u^\bot)$ and $t>0$.
Let us observe that
\[
\left(1-\frac1k\right)z_0+\frac1k(y_0+tku)\rightarrow z_0+tu\quad\text{ if }\quad k\rightarrow\infty.
\]
Furthermore, we have that
\[
f\left(\left(1-\frac1k\right)z_0+\frac1k(y_0+tku)\right) \geq f(z_0)^{1-\frac1k}f(y_0+tku)^\frac1k \geq \alpha^{1-\frac1k}c^\frac1k,
\]
and thus that
\[
f(z_0+tu)=\lim_{k\rightarrow\infty}f\left(\left(1-\frac1k\right)z_0+\frac1k(y_0+tku)\right)\geq \alpha.
\]
Note that $z_0+tu\in\mathrm{int}(\mathrm{supp }f)$.
Hence
\[
\int_{\mathbb R^n} f(x)dx\geq\int_{U+\mathbb R_+u}f(x)dx\geq \alpha\mathrm{vol}(U+\mathbb R_+u)=\infty,
\]
thus showing that $f$ is not integrable, a contradiction.

We now show (2) implies (1). If $\int f=0$, then $f$ is integrable. Let us suppose that $\int f>0$. After a suitable translation,
let us assume that $f(0)=\|f\|_\infty$. For every $u\in\mathbb S^{n-1}$, since $f$ is not constant on $\mathbb R_+u$, then
there exists $s_u>0$ such that $f(s_uu)<\|f\|_\infty$. Now, using that $f$ is continuous implies that $t_u=\inf\{s>0:f(su)<\Vert f\Vert_\infty\}$
fulfills $f(t_uu)=\Vert f\Vert_\infty$, for every $u\in\mathbb S^{n-1}$.
We now show that if $\{t_u:u\in\mathbb S^{n-1}\}$ is unbounded, we arrive at a contradiction.
Indeed, in that case let $\{u_k\}\subset\mathbb S^{n-1}$ be such that $t_{u_k}\rightarrow\infty$ as $k\rightarrow\infty$.
Since $\mathbb S^{n-1}$ is compact, there exists a subsequence (which we can suppose w.l.o.g.~to be the sequence itself)
converging $u_k\rightarrow u_0\in\mathbb S^{n-1}$. For every $t>0$, then
\[
\frac{t}{t_{u_k}}(t_{u_k}u_k)+\left(1-\frac{t}{t_{u_k}}\right)0\rightarrow tu_0\quad\text{if}\quad k\rightarrow\infty.
\]
Note that if $t>0$ is fixed, since $t_{u_k}\rightarrow\infty$, there exists $k_t\in\mathbb N$ such that if $k\geq k_t$
then $t_{u_k}\geq t$. Hence
\[
f(tu_0) = \lim_{k\rightarrow\infty}f\left(\frac{t}{t_{u_k}}(t_{u_k}u_k)+\left(1-\frac{t}{t_{u_k}}\right)0\right)
\geq \lim_{k\rightarrow\infty}f(t_{u_k}u_k)^{\frac{t}{t_{u_k}}}f(0)^{1-\frac{t}{t_{u_k}}}=\Vert f\Vert_\infty.
\]
Therefore $f|_{\mathbb R_+u_0}=\Vert f\Vert_\infty$, contradicting the hypothesis.
Thus $\{t_u:u\in\mathbb S^{n-1}\}$ is bounded. If $t>t_u$, then $f(tu)<\Vert f\Vert_\infty$.
Hence let $t_*>t_u$ be such that $f(t_*u)\leq c<\Vert f\Vert_\infty$ for every $u\in\mathbb S^{n-1}$ and some $c>0$.
Observe that for every $t\geq t_*$ and every $u\in \mathbb S^{n-1}$ we have that
\[
f(t_*u)=f\left(\frac{t_*}{t}(tu)+\left(1-\frac{t_*}{t}\right)0\right) \geq f(tu)^\frac{t_*}{t}f(0)^{1-\frac{t_*}{t}}.
\]
Thus, integrating in polar coordinates, we get that
\[
\begin{split}
\int_{\mathbb R^n}f(x)dx & = \int_{\mathbb R^n\setminus B(0,t_*)}f(x)dx+\int_{B(0,t_*)}f(x)dx \\
& = n\omega_n\int_{\mathbb S^{n-1}}\int_{t_*}^\infty t^{n-1}f(tu)dtd\mu(u) + \int_{B(0,t_*)}f(x)dx \\
& \leq n\omega_n\int_{\mathbb S^{n-1}}\int_{t_*}^\infty t^{n-1}f(0)\left(\frac{f(t_*u)}{f(0)}\right)^{\frac{t}{t_*}}dtd\mu(u) + \int_{B(0,t_*)}f(x)dx \\
& \leq n\omega_nf(0)\int_{t_*}^\infty t^{n-1}\left(\frac{c}{f(0)}\right)^\frac{t}{t_*}dt +\int_{B(0,t_*)}f(x)dx,
\end{split}
\]
where $\mu$ is the uniform probability measure in $\mathbb S^{n-1}$.
Since the first integral is finite as $0<c/f(0)<1$ and the second one is finite as $f$ is bounded and $B(0,t_*)$ is bounded too,
hence we obtain that $f$ is integrable.
\end{proof}

\begin{rmk}
    The existence of $x_M\in\mathbb R^n$ in Lemma \ref{lem:integrable} is necessary. Indeed, the function $f:\mathbb R\rightarrow[0,\infty)$ where $f(x):=e^{-e^x}$, which is monotonically decreasing on $\mathbb R$, is log-concave, it fulfills (2) (since it is not constant over any ray) but does not fulfill (1) (simply noticing that $f(x)\geq e^{-1}$ for every $x\leq 0$).
\end{rmk}

\begin{teo}\label{thm:IntegrPolar}
Let $f:\mathbb R^n\rightarrow\mathbb R$ be log-concave with $0\in\mathrm{int}(\mathrm{supp}(f))$ and $f^\circ(x_M)=\|f^\circ\|_\infty$ for some $x_M\in\mathbb R^n$. Then $f^\circ$ is integrable.
\end{teo}

\begin{proof}
Let us suppose that $f=e^{-\varphi}$. We show now that $\varphi^*$ (and thus $f^\circ$) is not
constant over any ray $y_0+\mathbb R_+u$, for any $y_0,u\in\mathbb R^n$, $u\neq 0$, thus concluding by Lemma \ref{lem:integrable} that $f^\circ$ is integrable.
Since $\varphi$ is continuous and $0\in\mathrm{int}(\mathrm{supp }f)$, let us suppose that $\varphi(x)\leq C$
whenever $|x|\leq\delta$ for some $C,\delta>0$. Let us consider $y_0+tu$, $t>0$. Then
\[
\varphi^*(y_0+tu)=\sup_{z}(\langle y_0+tu,z\rangle-\varphi(z))\geq\delta|y_0+tu|-\varphi\left(\delta\frac{y_0+tu}{|y_0+tu|}\right)\geq \delta|y_0+tu|-C,
\]
thus showing that $\varphi^*$ is not constant over any ray $y_0+\mathbb R_+u$, as desired.
\end{proof}

\begin{rmk}
The integrability of $f^\circ$ can be easily deduced by using some Blaschke-Santal\'o functional inequality
\[
\int f\int f^\circ\leq (2\pi)^n
\]
but only for certain particular translations of $f$ (for instance, when the Santal\'o point of $f$ is the origin, see \cite{artstein2004santalo}). However, the comment in \cite[Rmk.~2]{klartag2005geometry} is not correct
(where the authors said that "All of our results hold, with the same proofs, for log-concave
functions that reach their maximum at the origin"), since $f^\circ$ is not necessarily integrable if $f(0)=\|f\|_\infty$.  For instance, letting
\[
f(x)=\left\{\begin{array}{cc}e^{-\frac{x^2}{2}} & \text{if }x\geq 0 \\
0 & \text{otherwise,}\end{array}\right.
\]
if $f=e^{-\varphi}$, then
\[
\varphi(x)=\left\{\begin{array}{cc}\frac{x^2}{2} & \text{if }x\geq 0 \\
\infty & \text{otherwise}\end{array}\right.
\quad \text{and} \quad
\varphi^*(x)=\left\{\begin{array}{cc}\frac{x^2}{2} & \text{if }x\geq 0 \\
0 & \text{otherwise,}\end{array}\right.
\]
thus having that $f^\circ(x)=e^0=1$ if $x<0$, and hence $f^\circ$ would not be integrable.
\end{rmk}

\emph{Acknowledgements.} We would like to thank Juli\'an Haddad for his valuable comments and pointers, and to the University of Sevilla for hosting Leticia Alves da Silva during two months, in which this work has been done. She also thanks the support of the IFMG - Campus Bambu\'i while conducting this work.


\bibliographystyle{abbrv}
\bibliography{ref.bib}

\begin{thebibliography}{10}

\bibitem{alonso2017john}
D.~Alonso-Guti{\'e}rrez, B.~Gonz\'alez~Merino, C.~H. Jim{\'e}nez, and R.~Villa.
\newblock John’s ellipsoid and the integral ratio of a log-concave function.
\newblock {\em J. Geom. Anal.}, 28(2):1182--1201, 2018.

\bibitem{alonso2018zhang}
D.~Alonso-Guti\'errez, B.~G. Merino, and J.~Bernu\'es.
\newblock Zhang's inequality for log-concave functions.
\newblock In: Klartag B., Milman E. (eds) Geometric Aspects of Functional
  Analysis. Lecture Notes in Mathematics, vol 2256. Springer, Cham, 2020.

\bibitem{artstein2015asymptotic}
S.~Artstein-Avidan, A.~Giannopoulos, and V.~D. Milman.
\newblock {\em Asymptotic {G}eometric {A}nalysis, {P}art {I}}, volume 202.
\newblock American Mathematical Soc., 2015.

\bibitem{artstein2004santalo}
S.~Artstein-Avidan, B.~Klartag, and V.~Milman.
\newblock The {S}antal{\'o} point of a function, and a functional form of the
  {S}antal{\'o} inequality.
\newblock {\em Mathematika}, 51(1-2):33--48, 2004.

\bibitem{artstein2009The}
S.~Artstein-Avidan and V.~D. Milman.
\newblock The concept of duality in convex analysis and the characterization of
  the {L}egendre transform.
\newblock {\em Ann. of Math.}, 169(2):661--674, 2009.

\bibitem{bobkov2012reverse}
S.~Bobkov and M.~Madiman.
\newblock Reverse {B}runn--{M}inkowski and reverse entropy power inequalities
  for convex measures.
\newblock {\em J. Funct. Anal.}, 262(7):3309--3339, 2012.

\bibitem{colesanti2006functional}
A.~Colesanti.
\newblock Functional inequalities related to the rogers-shephard inequality.
\newblock {\em Mathematika}, 53(1):81--101, 2006.

\bibitem{cordero2015moment}
D.~Cordero-Erausquin and B.~Klartag.
\newblock Moment measures.
\newblock {\em J. Funct. Anal.}, 268(12):3834--3866, 2015.

\bibitem{fang2018lyz}
N.~Fang and J.~Zhou.
\newblock {LYZ} ellipsoid and {P}etty projection body for log-concave
  functions.
\newblock {\em Adv. Math.}, 340:914--959, 2018.

\bibitem{klartag2005geometry}
B.~Klartag and V.~D. Milman.
\newblock Geometry of log-concave functions and measures.
\newblock {\em Geom. Dedicata}, 112(1):169--182, 2005.

\bibitem{lutwak1993the}
E.~Lutwak.
\newblock The {B}runn--{M}inkowski--{F}irey {T}heory {I}: {M}ixed volumes and
  {M}inkowski {P}roblem.
\newblock {\em J. Differential Geom.}, 38(1):131--150, 1993.

\bibitem{petty1971isoperimetric}
C.~M. Petty.
\newblock Isoperimetric problems.
\newblock {\em Proceedings of the Conference on Convexity and Combinatorial
  Geometry}, 1:26--41, 1971.

\bibitem{rotem2012mean}
L.~Rotem.
\newblock On the mean width of log-concave functions.
\newblock In {\em Geom. Funct. Anal.}, pages 355--372. Springer, 2012.

\bibitem{schneider2014convex}
R.~Schneider.
\newblock {\em Convex bodies: the {B}runn--{M}inkowski {T}heory}.
\newblock Number 151. Cambridge university press, 2014.

\bibitem{zhang1991restricted}
G.~Zhang.
\newblock Restricted chord projection and affine inequalities.
\newblock {\em Geom. Dedicata}, 39(2):213--222, 1991.

\bibitem{zhang1999affine}
G.~Zhang.
\newblock The affine {S}obolev inequality.
\newblock {\em J. Differ. Geom.}, 53(1):183--202, 1999.

\end{thebibliography}

\end{document}